\documentclass[MSNbibl,number,citesort,seceqn,dvips]{arxbj}
\usepackage{upgreek}

\aid{0}
\volume{19}
\issue{5B}
\pubyear{2013}
\firstpage{2437}
\lastpage{2454}
\doi{10.3150/12-BEJ458} 

\makeatletter
\newremark{remark}[theorem]{Remark}

\newtheorem{theorem}{Theorem}[section]
\newtheorem{proposition}[theorem]{Proposition}
\newtheorem{corollary}[theorem]{Corollary}

\newcommand{\eqref}[1]{(\ref{#1})}
\newcommand{\arcosh}{\operatorname{ar\,cosh}}
\newcommand{\arsinh}{\operatorname{ar\,sinh}}
\makeatother

\begin{document}
\begin{frontmatter}

\title{On hyperbolic Bessel processes and beyond}
\runtitle{On hyperbolic Bessel processes and beyond}

\begin{aug}
\author{\fnms{Jacek} \snm{Jakubowski}\thanksref{e1}\ead[label=e1,mark]{jakub@mimuw.edu.pl}} \and
\author{\fnms{Maciej} \snm{Wi\'sniewolski}\corref{}\thanksref{e2}\ead[label=e2,mark]{wisniewolski@mimuw.edu.pl}}
\runauthor{J. Jakubowski and M. Wi\'sniewolski} 
\address{Institute of Mathematics, University
of Warsaw, Banacha 2, 02-097 Warszawa, Poland. \\\printead{e1,e2}}
\end{aug}

\received{\smonth{6} \syear{2011}}
\revised{\smonth{6} \syear{2012}}

%
\begin{abstract}
We investigate distributions of hyperbolic Bessel processes. We find
links between the hyperbolic cosine of hyperbolic Bessel
processes and functionals of geometric Brownian motion. We
present an explicit formula for the Laplace transform of the hyperbolic
cosine of a hyperbolic Bessel process and some other interesting
probabilistic representations of this Laplace transform.
We express the one-dimensional distribution of a hyperbolic Bessel
process in terms of other, known and independent processes. We
present some applications including a new proof of Bougerol's
identity and its generalization. We characterize the distribution
of the process which is the hyperbolic sine of hyperbolic Bessel
process.
\end{abstract}

%
\begin{keyword}
\kwd{Bessel process}
\kwd{Brownian motion}
\kwd{hyperbolic Bessel process}
\kwd{Laplace transform}
\end{keyword}

\end{frontmatter}

\section{Introduction}
The important role which functionals of Brownian motion play in many
fields of mathematics (including mathematical finance, diffusion processes
in random environment, probabilistic studies related to hyperbolic
spaces etc.) is a motivation to study the wide class of different
diffusion processes connected to those functionals somehow (see, e.g.,
\cite{SB,CPY,DGY,MatII} and
\cite{MatIII}). In this work, we consider a special class of such
diffusions: Brownian motion with a special stochastic drift. It is
well known in the literature that processes like Ornstein--Uhlenbeck processes,
Bessel processes or radial Ornstein--Uhlenbeck processes are examples of
Brownian motion with stochastic drift connected to special
functions: parabolic cylinder, Bessel or Kummer functions (see
\cite{B}). In this work, we consider another interesting class of
such diffusions -- hyperbolic Bessel processes which are connected to
Legendre functions. We establish the distribution of a hyperbolic
Bessel process. Borodin presented the computation of the transition
density function for a hyperbolic Bessel process (see formulas (5.4)
and (5.5) in \cite{B}). His method relied on the
connection between the Laplace transform of the transition density
function and two increasing and decreasing solutions of some
ordinary differential equation of the second order (for more
details of this method see also \cite{SB}). Gruet established the
transition density functions of hyperbolic processes by methods of
planar geometry
(see \cite{G,G2}). However, the results
obtained by both authors are technically very complicated. The role
of the hyperbolic Bessel processes in the world of Brownian motion
functionals was mentioned also by Byczkowski, Ma{\l}ecki and Ryznar
\cite{MBR}.

Our approach to hyperbolic Bessel processes is completely different
and purely probabilistic. We present results for hyperbolic Bessel
processes with index
$\alpha\geq-1/2$ and for fixed time. We find a link between
the hyperbolic cosine of a hyperbolic Bessel process and functionals of
geometric Brownian motion.
We present different probabilistic representations of the
Laplace transform of the hyperbolic cosine of a hyperbolic Bessel
process $R$.
We also give
probabilistic representations of the density of $\cosh R_t$.
The link between hyperbolic Bessel processes and functionals of
geometric Brownian motion enables us to establish a simple explicit
form of the Laplace transform of the vector $(\mathrm{e}^{B_t +
kt},\int_0^t\mathrm{e}^{2(B_u+ku)}\, \mathrm{d}u)$ for a standard Brownian motion $B$ and
a nonnegative integer $k$. These results are obtained for fixed $t$
(and not for stochastic time) and so they can be effectively
used in numerical computations.
We express the distribution of a hyperbolic Bessel process $R$ in
terms of
a squared Bessel process $X$ and a
vector $(B^{(\mu)}, A^{(\mu)})$ of Brownian
motion with drift and the integral of geometric Brownian motion,
which is independent of $X$. The joint distribution of $(B^{(\mu)},
A^{(\mu)})$ is known in the literature and is expressed via the density
function (e.g., see \cite{MatII}). As an interesting example of
applications, we find an alternative simple proof of the well known Bougerol
identity (for a thorough study, we refer the reader to \cite{ADY,AG} or \cite{Comoyo}).
We also characterize the distribution of $\sinh(R)$ for $R$ being
a hyperbolic Bessel process. It is interesting that this distribution
is deduced from the stochastic process which is in
a sense a generalization of a squared Bessel process.
At the end, we outline the case of exponential random time
independent of the Brownian motion driving the hyperbolic Bessel
process. In this case, many explicit and computable results are known
in the literature (see, e.g., \cite{SB,MatII} or \cite{MatIII}).

\section{Hyperbolic Bessel processes with fixed time}
We consider a complete probability space
$(\Omega,\mathcal{F},\mathbb{P})$ with filtration $\mathbb{F}=
(\mathcal{F}_t)_{t \in[0,\infty)}$ satisfying the usual conditions
and $\mathcal{F}=\mathcal{F}_{\infty}$, and with a standard Brownian
motion $B$. We define a hyperbolic
Bessel process $R$ of index $\alpha\in(-\frac12;\infty)$ as a nonnegative
diffusion, starting from a nonnegative $x$, given by
%
\begin{equation}
\label{eqhbp1} R_t = x + B_t + \biggl(\alpha+
\frac{1}{2} \biggr)\int_0^t\coth
R_u \, \mathrm{d}u.
\end{equation}
For $\alpha= -\frac12$, we define the hyperbolic Bessel process as the
reflected (at $0$) Brownian motion starting from a nonnegative $x$,
that is,
%
\begin{equation}
\label{eqhbp15} R_t = |x + B_t|.
\end{equation}

Now, we discuss properties of solutions to \eqref{eqhbp1} which follow
from general theory on SDE (see, e.g., Cherny and Engelbert \cite{cher-eng}).
The behavior of $R$ depends on the
index $\alpha$. For $\alpha\in(-\frac12,\infty)$, the point 0 is an
isolated singular point
of \eqref{eqhbp1}. For $\alpha\in(-\frac12,0)$ and $x\geq0$ there
exists a positive solution of \eqref{eqhbp1} defined up to $T_a= \inf\{
t\geq0\dvt   R_t=a\}$ for every
$a> x \geq0$ and such a solution is unique. This solution hits zero
with a positive probability (see \cite{cher-eng}, Theorem 2.12), but the point zero is instantaneously reflecting.
Indeed we check that $m(\{0\}) = 0$ for $m$ the speed measure of the
hyperbolic Bessel process (for the boundary classification see, e.g., Borodin and Salminen \cite{SB}, pages 14--16).
Therefore, by analogy with a SDE defining $\delta$-dimensional Bessel
process, in the case $\alpha\in(-\frac12,0)$ we understand \eqref
{eqhbp1} as
SDE
%
\begin{equation}
\label{eqhbp12} \mathrm{d}R_t =B_t + \biggl(\alpha+
\frac{1}{2} \biggr) \mathbb{\mathbf{1}}_{\{
R_t\neq0\}} \coth
R_t \, \mathrm{d}t ,\qquad R_0=x.
\end{equation}
For the case $\alpha\geq0$, we have two possibilities. If $x>0$, then
there exist a unique strictly positive solution defined up to $T_a$ for every
$a> x \geq0$.
If $x=0$, then there exists a positive solution of \eqref{eqhbp1}
defined up to $T_a$, and such a solution is unique. Moreover, $R>0$ on
$]0,T_a]$ (see \cite{cher-eng}, Theorem 2.16).
Summing up, general theory implies that for $\alpha> -\frac12$ and
$x\geq0$ there exists a unique positive solution defined up to the
explosion time.

Recall that Gruet \cite{G2} defines a hyperbolic Bessel
process of index $\alpha> -1/2$ as a nonnegative diffusion $R$
with generator
%
\begin{equation}
\mathcal{A} = \frac{1}{2}\frac{\mathrm{d}^2}{\mathrm{d}x^2} + \biggl(\alpha+
\frac
{1}{2} \biggr)\coth(x)\frac{\mathrm{d}}{\mathrm{d}x}.
\end{equation}
In the case of $-1/2 < \alpha<0$, the definition is completed by the
requirement that $0$ is an instantaneously reflecting point. It is
also assumed that the starting point is nonnegative. A similar
definition of hyperbolic Bessel process can be found in Borodin
\cite{B}. In \cite{G2}, by the identification of the Green function,
it is stated that a reasonable candidate for a hyperbolic Bessel
process with $\alpha= -1/2$ is a reflected Brownian motion (see the
end of Section 3 in \cite{G2}).

To start the discussion of the properties of hyperbolic Bessel
processes, we define a process $\theta$ to be the solution of the SDE
%
\begin{equation}
\label{SDETH} \mathrm{d}\theta_t = \sqrt{\bigl|\theta_t^2
- 1\bigr|}\, \mathrm{d}B_t + (\alpha+ 1)|\theta_t|\, \mathrm{d}t,
\end{equation}
where $\alpha\geq- 1/2$, and such that $\theta_0 = x\geq1$.
Observe that for $x\in\mathbb{R}$ and coefficients $\sigma(x) = \sqrt {|x^2-1|}$, $b(x) = (\alpha+ 1)|x|$ we have
\[
\bigl|\sigma(x)\bigr| + \bigl|b(x)\bigr| \leq(2+\alpha) \bigl(|x| +1\bigr),
\]
so there exists a weak solution, which does not explode. Using Ex. 3.14
\cite{RY}, Chapter IX (with $g\equiv1, c = 3, \delta= 1$), we see
that pathwise uniqueness holds for
\eqref{SDETH}. Thus, the SDE (\ref{SDETH}) has a unique strong nonexploding
solution.

Now, we consider the diffusion
$\eta_t := \theta_t - 1$. It is a diffusion with drift and diffusion
coefficients equal to $\tilde{b}(y) = (\alpha+1)|y+1|$ and
$\tilde{\sigma}(y) = \sqrt{|y^2+2y|}$,
respectively. We observe that 0 is an isolated singular point for $\eta
$. For $\alpha\geq
0$ and $a>0$, we have
\[
\int_0^a\exp \biggl(\int
_x^a\frac{2\tilde{b}(y)}{\tilde{\sigma^2}(y)} \, \mathrm{d}y \biggr)\, \mathrm{d}x = \int
_0^a \biggl(\frac{a^2+2a}{x^2+2x}
\biggr)^{\alpha+1}\, \mathrm{d}x \geq a^{\alpha+1}\int_0^ax^{-(\alpha+1)}\, \mathrm{d}x
= \infty.
\]
Therefore, from \cite{cher-eng}, Theorems 2.16 and 2.17, it follows that
$\eta_t> 0$ for all
$t>0$, provided $\eta_0> 0$. So, if $\theta_0 > 1$, then $\theta_t
> 1$ for all $t>0$.
Hence, the process $\theta$ satisfies for $\alpha\geq0$ the SDE
%
\begin{equation}
\label{SDETHH} \mathrm{d}\theta_t = \sqrt{\theta_t^2
- 1} \, \mathrm{d}B_t + (\alpha+ 1)\theta_t \, \mathrm{d}t.
\end{equation}
It turns out that in the case of $\alpha\geq0$ we can recover a
hyperbolic Bessel
process from the process~$\theta$. As usual, by $\arcosh$ we denote
the inverse function of $\cosh$.
%
\begin{theorem} \label{th-repHBP}
If $\alpha\geq0$ and $x > 1$, then the
process $R_t = \arcosh(\theta_t)$ is a hyperbolic Bessel
process of index $\alpha$,
where $\theta$ is the solution of \eqref{SDETHH}, $\theta_0 = x$.
\end{theorem}
\begin{pf} Observe that, by definition,
%
\begin{equation}
\label{conHB} R_t = \arcosh(\theta_t) = \ln
\Bigl(\theta_t + \sqrt{\theta_t^2-1}
\Bigr) > 0.
\end{equation}
As $\theta_t > 1$ for all $t\geq0$, we can use the It\^o
lemma to conclude that $R$ satisfies \eqref{eqhbp1}.
\end{pf}
%
\begin{theorem}
A hyperbolic Bessel process $R$ of index $\alpha\geq-1/2$ does not
explode.
\end{theorem}
\begin{pf}
For $\alpha\geq0$, we use Theorem
\ref{th-repHBP}, which implies that $R$ does not explode since
$\theta$ does not explode. For $\alpha= -\frac12$ the process $R$ is
given by \eqref{eqhbp15}, so it does not explode.
For $\alpha\in(-\frac12;0)$, we observe that
$\cosh(R)$ and $\theta$ have the same generator
$\mathcal{A}_{\theta}$ on $C_c^2$, the space of twice continuously
differentiable functions with compact support. From the uniqueness
of solution of the martingale problem induced by $\mathcal{A}_{\theta}$,
we conclude that $\cosh(R)$ and $\theta$ have the same
distribution. The diffusion $\theta$ does not explode, so $R$ does not
explode either.
\end{pf}
Now we give a formula for the Laplace transform of cosh of a hyperbolic
Bessel process with $\alpha\geq-\frac12$, one of the most important
results of this paper.
%
\begin{theorem} \label{twHB}
Fix $\alpha\geq-\frac12$. If $R$ is
a hyperbolic Bessel process of index $\alpha$ and $R_0 = x\geq0$, then
for $t>0$ and $\lambda> 0$,
%
\begin{equation}
\mathbb{E} \bigl[ \exp (-\lambda\cosh R_t ) \bigr] = \mathbb{E}
\biggl[ \exp \biggl(- \lambda\cosh(x) V_t -\frac{\lambda^2}{2}\int
_0^t V_u^2\, \mathrm{d}u \biggr)
\biggr],
\end{equation}
where $V_t = \mathrm{e}^{(\alpha+{1}/{2})t + B_t}$ and $B$ is a
standard Brownian motion.
\end{theorem}
\begin{pf}
Set $\theta_t = \cosh R_t \geq1$, so $\theta_0 = \cosh
x$. By the definition of $R$, the properties of the function $\cosh$
and the It\^o lemma we see that
%
\begin{equation}
\label{thetadyn} \mathrm{d}\theta_t = \sqrt{\theta_t^2
-1} \, \mathrm{d}B_t + a\theta_t\,\mathrm{d}t,
\end{equation}
where $a = \alpha+ 1$. Using again the It\^o lemma, we obtain
%
\begin{equation}
\label{eqTRH} \mathrm{d}\mathrm{e}^{-\lambda\theta_t} = -\lambda \mathrm{e}^{-\lambda\theta_t}\sqrt{
\theta_t^2 -1}\, \mathrm{d}B_t -\lambda a\mathrm{e}^{-\lambda\theta_t}
\theta_t\,\mathrm{d}t + \frac{\lambda^2}{2}\mathrm{e}^{-\lambda\theta_t}\bigl(
\theta_t^2 - 1\bigr)\, \mathrm{d}t.
\end{equation}
The process $\int_0^t \mathrm{e}^{-\lambda\theta_u}\sqrt{\theta_u^2
-1} \, \mathrm{d}B_u$ is a martingale
since $\mathrm{e}^{-x}\sqrt{x^2-1}\leq \mathrm{e}^{-1}$ for $x\geq 1$.
%

Thus, we infer  from (\ref{eqTRH}) that
%
\begin{equation}
\label{eqTRHH} \mathbb{E}\mathrm{e}^{-\lambda\theta_{t}} = \mathrm{e}^{-\lambda\theta_0} -a\lambda\int
_0^t \mathbb{E}\bigl[\mathrm{e}^{-\lambda\theta_{u}}
\theta_{u}\bigr]\, \mathrm{d}u + \frac{\lambda^2}{2}\int_0^t
\mathbb{E}\bigl[\mathrm{e}^{-\lambda\theta_{u}}\bigl(\theta^2_{u}-1\bigr)
\bigr]\, \mathrm{d}u.
\end{equation}
Define $p(t,\lambda):= \mathbb{E}\mathrm{e}^{-\lambda\theta_{t}}$, so the
function $p$ is bounded. Using (\ref{eqTRHH}), we deduce that $p$
satisfies the following partial differential equation:
%
\begin{equation}
\label{eqrrp} \frac{\partial p}{\partial t} = -\frac{\lambda^2}{2}p + a\lambda
\frac{\partial p}{\partial\lambda} + \frac{\lambda^2}{2}\frac
{\partial^2 p}{\partial\lambda^2},
\end{equation}
with $p(0,\lambda) = \mathrm{e}^{-\lambda\theta_0}$. Set $V_t =
V_0\mathrm{e}^{(\alpha+{1}/{2}) t + B_t}$. Then
\[
\mathrm{d}V_t = V_t\,\mathrm{d}B_t + aV_t\,\mathrm{d}t
\]
with $a=\alpha+1$. The generator of $V$ is
\[
\mathcal{A}_V = \frac{x^2}{2}\frac{\mathrm{d}^2}{\mathrm{d}x^2} + ax
\frac{\mathrm{d}}{\mathrm{d}x}.
\]
Observe, by \eqref{eqTRHH}, that $p\in
C^{1,2}((0,\infty)\times(0,\infty))$
and from the
Feynman--Kac theorem (see \cite{KS}, Chapter 5, Theorem 7.6) we know
that the only bounded solution of the partial differential equation
\[
\frac{\partial u}{\partial t} = \mathcal{A}_V u - \frac{x^2}{2}u,\qquad u(0,x)
= \mathrm{e}^{-\theta_0 x},
\]
that is, \eqref{eqrrp}, admits the stochastic representation
\[
u(t,x) = \mathbb{E} \exp \biggl(-\theta_0V_t-
\frac{1}{2}\int_0^t V_u^2\, \mathrm{d}u
\biggr),
\]
where $V_0 = x$. This implies the assertion of the theorem.
\end{pf}
As an easy consequence, we obtain
the following proposition.
%
\begin{proposition} \label{VBM} The Laplace transform of the vector
$(\mathrm{e}^{B_t},\int_0^t\mathrm{e}^{2B_u}\, \mathrm{d}u)$ is given by
%
\begin{equation}
\label{eqVBM} \mathbb{E}\exp \biggl(-\gamma \mathrm{e}^{B_t} - \frac{\lambda^2}{2}
\int_0^t\mathrm{e}^{2B_u}\, \mathrm{d}u \biggr) =
\mathbb{E}\exp \bigl(-\lambda\cosh(x + B_t) \bigr)
\end{equation}
for $\gamma>0$ and $\lambda>0$, where $x = \arcosh
\frac{\gamma}{\lambda}$.
\end{proposition}
\begin{pf}
The conclusion follows by applying Theorem \ref{twHB} with $\alpha=
-\frac{1}{2}$,
and the observation that $\cosh(R_t)=\cosh(x + B_t)$, since $\cosh$ is
an even function.
This implies that for any $\lambda>0$,
\[
\mathbb{E}\exp \bigl(-\lambda\cosh(x + B_t) \bigr) = \mathbb{E}\exp
\bigl(-\lambda\cosh(R_t) \bigr) .
\]
\upqed\end{pf}

\begin{remark}
The form of the Laplace transform of the vector
$(\mathrm{e}^{B_t},\int_0^t\mathrm{e}^{2B_u}\, \mathrm{d}u)$ in the last proposition, that is,
\eqref{eqVBM}, is very simple when compared to the form of
density of this vector obtained by Matsumoto and Yor \cite{MatII}.
Indeed, the density given in \cite{MatII} has the oscillating
nature in the neighbourhood of $0$ and is not convenient for
computational use (see \cite{BRY}). The knowledge of the Laplace
transform makes it possible to invert it numerically and obtain the
density of $(\mathrm{e}^{B_t},\int_0^t\mathrm{e}^{2B_u}\, \mathrm{d}u)$. It is also important to
realize that it enables one to obtain the density for $(\mathrm{e}^{B^{(\mu)}_t},
\int_0^t\mathrm{e}^{2B^{(\mu)}_u}\,\mathrm{d}u)$, where $B^{(\mu)}_t := B_t +\mu t$,
$\mu\in\mathbb{R}$, as we know from \cite{MatII} that
\[
\mathbb{P} \biggl(B^{(\mu)}_t \in \mathrm{d}x , \int
_0^t\mathrm{e}^{2B^{(\mu)}_u}\, \mathrm{d}u\in \mathrm{d}y \biggr) =
\mathrm{e}^{\mu x -\mu^2t / 2}\mathbb{P} \biggl(B_t \in \mathrm{d}x , \int
_0^t\mathrm{e}^{2B_u}\, \mathrm{d}u\in \mathrm{d}y \biggr).
\]
\end{remark}
%
\begin{theorem} \label{twHBN}
Let $k\in\mathbb{N}$, $\gamma>0$, $\lambda>0$ and $R$ be a
hyperbolic Bessel process
%
\begin{equation}
R_t = \arcosh(\gamma/\lambda) + B_t + k\int
_0^t\coth R_u \, \mathrm{d}u.
\end{equation}
For $\lambda> 0$, we have
\[
\mathbb{E}\exp (-\lambda\cosh R_t ) =(-1)^k
\mathrm{e}^{({k^2}/{2})t}\frac{\partial^{k}p(\gamma,\lambda
)}{\partial\gamma^k},
\]
where
\[
p(\gamma,\lambda) = \mathbb{E}\mathrm{e}^{-\lambda\cosh(\arcosh(\gamma
/\lambda)
+B_t)}.
\]
\end{theorem}
\begin{pf}
By Proposition
\ref{VBM}, for $x = \arcosh(\gamma/\lambda)$,
\[
p(\gamma,\lambda) = \mathbb{E} \mathrm{e}^{-\lambda\cosh(x + B_t)} = \mathbb{E}\exp \biggl(-\gamma
\mathrm{e}^{B_t} - \frac{\lambda^2}{2}\int_0^t\mathrm{e}^{2B_u}\, \mathrm{d}u
\biggr).
\]
Define the new probability measure $\mathbb{Q}$ by
\[
\frac{d\mathbb{Q}}{d\mathbb{P}} \bigg|_{\mathcal{F}_t} = \mathrm{e}^{k B_t - ({k^2}/{2})t}.
\]
Then $V_t = B_t - k t$ is a standard Brownian motion under
$\mathbb{Q}$ and
\begin{eqnarray*}
p(\gamma,\lambda)&=& \mathbb{E}_{\mathbb{Q}} \exp \biggl(-k B_t +
\frac
{k^2}{2}t-\gamma \mathrm{e}^{B_t} - \frac{\lambda^2}{2}\int
_0^t\mathrm{e}^{2B_u}\, \mathrm{d}u \biggr)
\\
&=& \mathrm{e}^{({k^2}/{2})t}\mathbb{E}_{\mathbb{Q}}\exp \biggl(-k (V_t +k
t) -\gamma \mathrm{e}^{(V_t +k t)} - \frac{\lambda^2}{2}\int_0^t
\mathrm{e}^{2(V_u +k u)}\, \mathrm{d}u \biggr).
\end{eqnarray*}
The result follows from Theorem \ref{twHB} after taking the $k$th
derivative of $p$ with respect to~$\gamma$.
\end{pf}
%
\begin{proposition} If $k\in\mathbb{N}$ and $Y_t = \mathrm{e}^{B_t + k t}$, then
for $\lambda>0, \gamma>0$,
\[
\mathbb{E}\exp \biggl(-\gamma Y_t - \frac{\lambda^2}{2} \int
_0^tY_u^2\, \mathrm{d}u \biggr) =
(-1)^k\mathrm{e}^{({k^2}/{2})t} \frac{\partial^{k}p(\gamma,\lambda)}{\partial
\gamma^k},
\]
where $p(\gamma,\lambda)$ is as in Theorem \ref{twHBN}.
\end{proposition}
\begin{pf}
This follows from Theorems \ref{twHBN} and \ref{twHB}.
\end{pf}

\begin{theorem} \label{twHBB}
If $R$ is a hyperbolic Bessel process
of index $\alpha\geq- \frac12$ starting from $x$,
then for $\lambda> 0$,
\[
\mathbb{E}\exp (-\lambda\cosh R_t ) = \frac{\sqrt{2\uppi}}{\sqrt{t}}\mathrm{e}^{-({a^2t}/{2})}
\mathbb{E} \bigl(1_{\{V_t\geq|B_t|\}}V_t h_t(B_t,x)J_0
\bigl(\lambda\phi (B_t,V_t)\bigr) \bigr),
\]
where $a = \alpha+\frac{1}{2}$, $J_0$ is the Bessel function of the
first kind of order $0$, $h_t(z,x) = \exp(\frac{z^2}{2t} + az -
\lambda\cosh(x) \mathrm{e}^{z}) $, $\phi(v,z) = \sqrt{2}\mathrm{e}^{v/2}(\cosh z - \cosh
v)^{1/2}$ for $z \geq|v|$, and $V,B$ are two independent standard
Brownian motions.
\end{theorem}
\begin{pf} Let $W_t = B_t + at$ and let $\mathbb{Q}$ be the new
probability measure given by
\[
\frac{d\mathbb{Q}}{d\mathbb{P}} \bigg|_{\mathcal{F}_t} = \mathrm{e}^{-a B_t -
({a^2}/{2})t}.
\]
From Theorem \ref{twHB}, we have
\begin{eqnarray*}
\mathbb{E}\exp (-\lambda\cosh R_t ) &=& \mathbb{E}\exp \biggl(-
\lambda\cosh(x) \mathrm{e}^{at + B_t} -\frac{\lambda^2}{2}\int_0^t\mathrm{e}^{2(B_u+au)}\, \mathrm{d}u
\biggr)
\\
&=& \mathbb{E}_{\mathbb{Q}}\exp \biggl(aW_t-\frac{a^2}{2}t -
\lambda\cosh (x) \mathrm{e}^{W_t} -\frac{\lambda^2}{2}\int_0^t\mathrm{e}^{2W_u}\, \mathrm{d}u
\biggr)
\\
&=& \mathrm{e}^{-({a^2}/{2})t}\mathbb{E}\exp \biggl(aB_t - \lambda\cosh(x)
\mathrm{e}^{B_t} -\frac{\lambda^2}{2}\int_0^t\mathrm{e}^{2B_u}\, \mathrm{d}u
\biggr),
\end{eqnarray*}
so
\[
\mathbb{E}\exp (-\lambda\cosh R_t ) = \mathrm{e}^{-({a^2}/{2})t}\mathbb {E}
\bigl[\mathrm{e}^{aB_t -
\lambda\cosh(x)\mathrm{e}^{B_t}}\mathbb{E} \bigl(\mathrm{e}^{ -({\lambda^2}/{2})\int
_0^t\mathrm{e}^{2B_u}\, \mathrm{d}u}|B_t \bigr)
\bigr].
\]
To finish the proof, we use the conditional Laplace transform (see
(5.5) in \cite{MatII})
\[
\mathbb{E} \bigl(\mathrm{e}^{ -({\lambda^2}/{2})\int_0^t\mathrm{e}^{2B_u}\, \mathrm{d}u}|B_t = x \bigr)
\frac{1}{\sqrt{2\uppi t}}\mathrm{e}^{-{x^2}/({2t})} = \int_{|x|}^{\infty}
\frac{z}{\sqrt{2\uppi t^3}}\mathrm{e}^{-
{z^2}/({2t})}J_0\bigl(\lambda\phi(x,z)\bigr)\, \mathrm{d}z.
\]
\upqed\end{pf}
Now we define a process $\Gamma^{(\alpha)}$.
Let $B$ be a standard Brownian motion, $\alpha\in\mathbb{R}$,
$\lambda> 0$ and
%
\begin{equation}
\label{defal} \Gamma_t^{(\alpha)} = \frac{\mathrm{e}^{B_t + \alpha t}}{1+\lambda\int_0^t\mathrm{e}^{B_u+\alpha u}\, \mathrm{d}u}.
\end{equation}
%
\begin{theorem}
\label{twGALP}
If $R$ is a hyperbolic Bessel process of index $\alpha\geq- \frac12$
starting from $x$,
then for $\lambda> 0$,
\[
\mathbb{E}\exp (-\lambda\cosh R_t ) = \mathrm{e}^{-\lambda}\mathbb{E}
\biggl[\mathrm{e}^{-\lambda(\cosh(x)-1)\Gamma_t^{(\alpha
+1/2)}} \biggl(1+\lambda\int_0^t\mathrm{e}^{B_u+(\alpha+1/2) u}\, \mathrm{d}u
\biggr)^{-\alpha
-1} \biggr],
\]
where $B$ is a standard Brownian motion and $\Gamma_t^{(\alpha)}$ is
given by \eqref{defal}.
\end{theorem}
\begin{pf} Set $Y_t^{(\alpha)} = \mathrm{e}^{B_t +\alpha t}$. Define the new
measure $\mathbb{Q}$ by
\[
\frac{d\mathbb{Q}}{d\mathbb{P}} \bigg|_{\mathcal{F}_t} = \mathrm{e}^{-\lambda\int
_0^tY^{(\alpha+1/2)}_u\,\mathrm{d}B_u - ({\lambda^2}/{2})\int_0^t(Y^{(\alpha
+1/2)}_u)^2\, \mathrm{d}u}.
\]
Theorem \ref{twHB} implies
\begin{eqnarray*}
\mathbb{E}\mathrm{e}^{-\lambda\cosh R_t} &=& \mathbb{E} \bigl[\mathrm{e}^{-\cosh(x)
\lambda Y^{(\alpha+1/2)}_t -({\lambda^2}/{2})\int_0^t(Y^{(\alpha
+1/2)}_u)^2\, \mathrm{d}u} \bigr]
\\
&=& \mathbb{E}_{\mathbb{Q}} \bigl[\mathrm{e}^{-\cosh(x) \lambda Y^{(\alpha+1/2)}_t
+ \lambda\int_0^tY^{(\alpha+1/2)}_u\,\mathrm{d}B_u} \bigr]
\\
&=& \mathbb{E}_{\mathbb{Q}} \bigl[\mathrm{e}^{-\cosh(x) \lambda Y^{(\alpha+
1/2)}_t +
\lambda(Y^{(\alpha+1/2)}_t-1) -
(\alpha+1)\lambda\int_0^tY^{(\alpha+1/2)}_u\,\mathrm{d}u} \bigr]:= I,
\end{eqnarray*}
where we have used the fact that
%
\begin{equation}
Y_t^{(\alpha+1/2)} = 1 + \int_0^tY^{(\alpha+1/2)}_u\,\mathrm{d}B_u
+ (\alpha+1)\int_0^tY^{(\alpha+1/2)}_u\,\mathrm{d}u.
\end{equation}
From the Girsanov theorem, the process $V_t = B_t +
\lambda\int_0^tY^{(\alpha+1/2)}_u\,\mathrm{d}u$ is a standard Brownian
motion under $\mathbb{Q}$. By the result of Alili, Matsumoto, and
Shiraishi \cite{AMY}, Lemma 3.1, we have
%
\begin{equation}
B_t = V_t - \ln \biggl(1 +\lambda\int
_0^t\mathrm{e}^{V_u + (\alpha+1/2)u}\, \mathrm{d}u \biggr).
\end{equation}
Thus,
\[
Y^{(\alpha+1/2)}_t =\frac{\mathrm{e}^{V_t + (\alpha+1/2)
t}}{1+\lambda\int_0^t\mathrm{e}^{V_u+(\alpha+1/2) u}\, \mathrm{d}u},
\]
and
\begingroup
\abovedisplayskip=7pt
\belowdisplayskip=7pt
\begin{eqnarray*}
I &=& \mathrm{e}^{-\lambda}\mathbb{E}_{\mathbb{Q}} \biggl[\exp\biggl({-\lambda\bigl(\cosh
(x)-1\bigr)\frac{\mathrm{e}^{V_t +
(\alpha+1/2) t}}{1+\lambda\int_0^t\mathrm{e}^{V_u+(\alpha+1/2) u}\, \mathrm{d}u}-
(\alpha+1)(V_t-B_t)}\biggr) \biggr]
\\
&=& \mathrm{e}^{-\lambda}\mathbb{E} \biggl[\mathrm{e}^{-\lambda(\cosh(x)-1)\Gamma_t^{(\alpha
+1/2)}} \biggl(1+\lambda \int
_0^t\mathrm{e}^{B_u+(\alpha+1/2) u}\, \mathrm{d}u \biggr)^{-\alpha-1}
\biggr].
\end{eqnarray*}
\upqed\end{pf}

Let us observe that for a hyperbolic Bessel process of index $\alpha
\geq- \frac12$ starting from $x=0$ the expectation $\mathbb{E}[\exp
(-\lambda\cosh R_t)]$ involves the random variable
$\int_0^t\mathrm{e}^{B_u+(\alpha+1/2) u}\, \mathrm{d}u$ and not $\Gamma_t^{(\alpha)}$.

Taking $\alpha= - 1/2$ in the previous result, we obtain
the following corollary.
%
\begin{corollary}
For $\lambda> 0$, $x\geq0$,
\[
\mathbb{E} \bigl[\exp \bigl(-\lambda\cosh(B_t+x) \bigr) \bigr] =
\mathrm{e}^{-\lambda}\mathbb{E} \biggl[\mathrm{e}^{-\lambda(\cosh(x)-1)\Gamma_t^{(0)} } \biggl(1+\lambda\int
_0^t \mathrm{e}^{B_u}\, \mathrm{d}u \biggr)^{-1/2}
\biggr].
\]
\end{corollary}

We will use the following notation:
%
\begin{equation}
\label{wzor-na-A} A_t^{(\alpha)} = \int_0^t
\mathrm{e}^{2(B_u + \alpha u)}\, \mathrm{d}u,\qquad \alpha\in \mathbb{R},\qquad A_t=A_t^{(0)},
\end{equation}
with $B$ being a
standard Brownian motion.
%
\begin{theorem}\label{Beshalf}
If $R$ is a hyperbolic Bessel process
of index $\alpha=0$ and $x=0$,
then the density of $\cosh R_t$ on $[1, \infty)$, $t>0$, has the form
%
\begin{equation}
\mathbb{P}(\cosh R_t\in \mathrm{d}z) = \mathbb{E} \biggl[ \frac{1}{4A_{t/4}^{(1)}}
\exp \biggl(-\frac{z-1}{4A_{
t/4}^{(1)}} \biggr) \biggr]\, \mathrm{d}z.
\end{equation}
\end{theorem}
\begin{pf} Fix $t>0$.
Using Theorem \ref{twGALP} for $\alpha= x = 0$ and $\lambda> 0$ we obtain
\[
\mathbb{E}\mathrm{e}^{-\lambda(\cosh R_t-1)} = \mathbb{E} \biggl[\frac{1}{1+\lambda\int_0^t\mathrm{e}^{B_u
+{u}/{2}}\, \mathrm{d}u} \biggr].
\]
From the last equality and the scaling property of Brownian motion, we have
%
\begin{equation}
\label{Hyperzero} \mathbb{E}\mathrm{e}^{-\lambda(\cosh R_{4t}-1)} = \mathbb{E} \biggl[
\frac
{1}{1+4\lambda\int_0^t\mathrm{e}^{B_{4u} +2u}\, \mathrm{d}u} \biggr] = \mathbb{E} \biggl[\frac{1}{1+4\lambda A_t^{(1)}} \biggr].
\end{equation}
From Theorem 2.8 in \cite{jw}, we know that
\[
\mathbb{E} \biggl[\frac{1}{1+4\lambda A_t^{(1)}} \biggr] = 1 - 4\lambda\int
_0^{\infty}G_t(y)\mathrm{e}^{-4\lambda
y}\, \mathrm{d}y,
\]
where $G_t(y) = \mathbb{E} (\mathrm{e}^{-y/A_t^{(1)}} )$.
\endgroup

Since $G_t$ is differentiable,
integration by parts yields
\[
\mathbb{E} \biggl[\frac{1}{1+4\lambda A_t^{(1)}} \biggr]= 1 + \int_0^{\infty}G_t(y)
\bigl( \mathrm{e}^{-4\lambda y}\bigr)' \,\mathrm{d}y = -\int_0^{\infty}G_t'(y)\mathrm{e}^{-4\lambda y}\,\mathrm{d}y.
\]
From the last equality and the definition of $G_t$, we have
%
\begin{eqnarray}
\label{InvMom} \mathbb{E} \biggl[\frac{1}{1+4\lambda A_t^{(1)}} \biggr]&=& \int
_0^{\infty
}\mathrm{e}^{-4\lambda y}\mathbb{E} \biggl[
\frac{1}{A_t^{(1)}}\exp \biggl(-\frac{y}{A_t^{(1)}} \biggr) \biggr] \,\mathrm{d}y
\nonumber
\\[-8pt]\\[-8pt]
&= &\int_0^{\infty}\mathrm{e}^{-\lambda
z}\mathbb{E}
\biggl[\frac{1}{4A_t^{(1)}}\exp \biggl(-\frac{z}{4A_t^{(1)}} \biggr) \biggr]\,\mathrm{d}z.\nonumber
\end{eqnarray}
From (\ref{Hyperzero}) and (\ref{InvMom}), we conclude that
\[
\mathbb{P} \bigl((\cosh R_{4t}-1)\in \mathrm{d}z \bigr) = \mathbb{E} \biggl[
\frac
{1}{4A_t^{(1)}}\exp \biggl(-\frac{z}{4A_t^{(1)}} \biggr) \biggr]\,\mathrm{d}z
\]
for $z\geq0$, which ends the proof.
\end{pf}
%
\begin{remark}
Using the explicit form of the density of $A_t^{(1)}$ and Theorem
\ref{Beshalf}, we can obtain the integral form of the density of $\cosh
R_t$, with $\alpha=0$ and $x=0$.
\end{remark}
%
\begin{proposition}
If $R$ is a hyperbolic Bessel process of index
$\alpha=0$ and $x=0$,
then on $[1,\infty)$,
%
\begin{equation}
\mathbb{P}(\cosh R_t\in \mathrm{d}z) = -\frac{1}{4}G'_{t/4}
\biggl(\frac{z
{-1}}{4} \biggr)\,\mathrm{d}z
\end{equation}
for $t>0$, where
%
\begin{equation}
\label{wzor-g} G_t(y) = \mathrm{e}^{-t/2}\mathbb{E}\exp
\biggl(B_t +\frac{1}{2t} \bigl(B_t^2 -
\varphi^2_y(B_t) \bigr) \biggr),
\end{equation}
$B$ is a standard Brownian motion and
%
\begin{equation}
\label{varphi} \varphi_y(z) = \ln \bigl(y\mathrm{e}^{-z} +
\cosh(z)+\sqrt{y^2\mathrm{e}^{-2z}+\sinh^2(z)
+2y\mathrm{e}^{-z}\cosh(z)} \bigr).
\end{equation}
\end{proposition}
\begin{pf}
As before, let $G_t(y) =
\mathbb{E} (\mathrm{e}^{-y/A_t^{(1)}} )$. As an easy
consequence of the Matsumoto--Yor result \cite{MatII}, Theorem 5.6, we
find that $G_t$ is given by \eqref{wzor-g} (see the proof of
Theorem 2.5 in \cite{jw}). Now the assertion follows from Theorem
\ref{Beshalf} and the observation that
\[
\mathbb{E} \biggl[\frac{1}{4A_{t/4}^{(1)}}\exp \biggl(-\frac{z-1}{4A_{
t/4}^{(1)}} \biggr)
\biggr] = -\frac{1}{4} G'_{t/4} \biggl(
\frac{z {-1}}{4} \biggr).
\]
\upqed\end{pf}
Let us recall that the formula of Laplace transform of a squared Bessel process
$X$ of index $\alpha$ has the form
%
\begin{equation}
\label{eqBESQL} \mathbb{E}\mathrm{e}^{-\lambda X_t} = (1 + 2\lambda t)^{-(\alpha+1)}\exp
\bigl(-\lambda x/(1+2\lambda t)\bigr)
\end{equation}
for $\lambda> 0, t\geq0$ and $X_0 = x\geq0$ (see \cite{RY}, Chapter XI,
page 441).
We will use \eqref{eqBESQL} in the proof of the next theorem and subsequently.

The next theorem states that, for a fixed $t$, that $R_t \stackrel{\mathrm{(law)}} =
F(B,A,X)$ for some functional $F$, so the distribution of $R_t$ can
be represented as a functional of $A, B$ and a squared Bessel process
$X$ of index $\alpha$ independent of $B$.
%
\begin{theorem}\label{DHB} If $R$ is a hyperbolic Bessel process of
index $\alpha\geq- \frac12$, $R_0=x\geq0$,
then
%
\begin{equation}
\cosh(R_t) \stackrel{\mathrm{(law)}} = 1 + 2 A_{t/4}^{(2\alpha+1)}X_1
\end{equation}
for every $t>0$, where $X$ is a squared Bessel process of index
$\alpha$, independent of a standard Brownian motion $B$ and
starting from $\frac12(\cosh(x) -1)\mathrm{e}^{2B^{(2\alpha
+1)}_{t/4}}/A_{t/4}^{(2\alpha+1)}$.
Moreover, we have
%
\begin{equation}
R_t \stackrel{\mathrm{(law)}} = \arcosh \bigl(1 + 2
A_{t/4}^{(2\alpha+1)}X_1 \bigr).
\end{equation}
\end{theorem}
\begin{pf} From Theorem \ref{twGALP} and the form of the Laplace
transform of a squared Bessel process, that is, \eqref{eqBESQL},
we have
\begin{eqnarray*}
\mathbb{E}\exp (-\lambda\cosh R_t ) &=& \mathrm{e}^{-\lambda}\mathbb{E}
\biggl[\mathrm{e}^{-\lambda(\cosh(x)-1)\Gamma_t^{(\alpha
+1/2)}} \biggl(1+\lambda\int_0^t\mathrm{e}^{B_u+(\alpha+1/2) u}\,\mathrm{d}u
\biggr)^{-\alpha-1} \biggr]
\\
&=& \mathrm{e}^{-\lambda}\mathbb{E}\exp \bigl(-\lambda \mathrm{e}^{B_t+(\alpha+1/2)t}
\hat{X}_{\int_0^t\mathrm{e}^{B_u+(\alpha+1/2) u}\,\mathrm{d}u/(2\mathrm{e}^{B_t+(\alpha+1/2)t})} \bigr)
\end{eqnarray*}
for arbitrary
$\lambda>0$, where $\hat{X}$ is a squared Bessel process of index
$\alpha$,
independent of a standard Brownian motion $B$ and starting from $
\cosh(x) -1$.
It is now clear that
%
\begin{equation}
\label{eqlicz} \cosh(R_t) \stackrel{\mathrm{(law)}} = 1 + \mathrm{e}^{B_t+(\alpha+1/2)t}
\hat{X}_{\int_0^t\mathrm{e}^{B_u+(\alpha+1/2) u}\,\mathrm{d}u/(2\mathrm{e}^{B_t+(\alpha+1/2)t})}.
\end{equation}
Using the scaling property of a squared Bessel process (see
\cite{RY}, Chapter XI, Proposition 1.6), we obtain
%
\begin{equation}
\label{eqlicz1} 1 + \mathrm{e}^{B_t+(\alpha+1/2)t}\hat{X}_{\int_0^t\mathrm{e}^{B_u+(\alpha+1/2) u}\,\mathrm{d}u/(2\mathrm{e}^{B_t+
(\alpha+1/2)t})} = 1 + \frac{1}{2} {
\overline X}_1\int_0^t\mathrm{e}^{B_u+(\alpha+1/2) u}\,\mathrm{d}u,
\end{equation}
where $\overline X$ is a squared Bessel process of index $\alpha$
independent of $B$ and starting from the stochastic point
${\overline X}_0 = 2(\cosh(x) -1)\mathrm{e}^{B_t+(\alpha
+1/2)t}/\int_0^t\mathrm{e}^{B_u +(\alpha+1/2) u}\,\mathrm{d}u$. Hence, taking
$4t$ instead of $t$, we infer that
\[
\cosh(R_{4t}) \stackrel{\mathrm{(law)}} = 1 + \frac{1}{2} {\overline
X}_1\int_0^{4t}\mathrm{e}^{B_u+(\alpha+1/2) u}\,\mathrm{d}u
\stackrel{\mathrm{(law)}} = 1 + 2{\overline A}_t^{(2\alpha+1)} {
X}_1,
\]
where $X$ is a squared Bessel process of index $\alpha$ independent
of a standard Brownian motion ${\overline B}$ and starting from
${X}_0 = \frac12(\cosh(x) -1)\mathrm{e}^{2{\overline B}^{(2\alpha
+1)}_t}/{\overline A}_t^{(2\alpha+1)}$, where ${\overline
A}_t^{(2\alpha+1)}$ is given by \eqref{wzor-na-A} with ${\overline
B}$ instead of $B$. As $R_t\geq0$, the second part of the theorem
follows from the first one.
The proof is complete.
\end{pf}
%
\begin{remark} Fix $\alpha\geq-1/2$ and $x\geq0$. For every $t>0$,
from the proof of Theorem \ref{DHB}
we deduce that
%
\begin{equation}
\label{rsa} R_t \stackrel{\mathrm{(law)}} = \arcosh \bigl(1 +
(1/2) a_{t}X_1 \bigr),
\end{equation}
where $a_t= \int_0^t\mathrm{e}^{B_u+(\alpha+1/2) u}\,\mathrm{d}u$, $X$ is a squared
Bessel process of index $\alpha\geq-1/2$ independent of a standard Brownian
motion $B$ and starting from the random point ${X}_0 = 2(\cosh(x)
-1)\mathrm{e}^{B_t+(\alpha+1/2)t}/\int_0^t\mathrm{e}^{B_u +(\alpha+1/2)
u}\,\mathrm{d}u$ (see \eqref{eqlicz} and \eqref{eqlicz1}).
\end{remark}
Theorem \ref{DHB}, in the special case $x=0$, gives
%
\begin{proposition} \label{zero} \emph{(a)} If $x=0$ and $\alpha\geq- 1/2$,
then for every $t$ the density function of $\cosh(R_t)$ on
$[1,\infty)$ is
%
\begin{equation}
\label{gestosc-cosh} \mathbb{P}\bigl(\cosh(R_t)\in \mathrm{d}z\bigr) =
\frac{1}{4^{\alpha+1}} \frac{(z-1)^{\alpha}}{\Gamma(\alpha+1)} \mathbb{E} \biggl[ \mathrm{e}^{-(z-1)/(4A_{t/4}^{(2\alpha+1)})}
\frac{1}{
(A_{t/4}^{(2\alpha+1)} )^{\alpha+1}} \biggr] \,\mathrm{d}z.
\end{equation}

\emph{(b)} If $x=0$ and $\alpha\geq- 1/2$, then for every $t$ the density
function of $R_t$ on
$[0,\infty)$ is
\[
\mathbb{P}(R_t\in \mathrm{d}z) = \frac{1}{4^{\alpha+1}} \frac{(\cosh(z)-1)^{\alpha}\sinh(z)}{\Gamma(\alpha+1)}
\mathbb{E} \biggl[ \mathrm{e}^{-(\cosh(z)-1)/(4A_{t/4}^{(2\alpha+1)})}\frac
{1}{ (A_{t/4}^{(2\alpha+1)} )^{\alpha+1}} \biggr] \,\mathrm{d}z.
\]
\end{proposition}
\begin{pf} From Theorem \ref{DHB} for $x=0$ and $\alpha\geq- 1/2$,
we obtain
%
\begin{equation}
\label{eqpostacR} \cosh(R_t) \stackrel{\mathrm{(law)}} = 1 + 2
A_{t/4}^{(2\alpha
+1)}X_1 ,
\end{equation}
where $X$ is a squared Bessel process of index $\alpha$, independent
of $B$ and starting from $0$.

(a) \eqref{eqpostacR} implies,
for $x=0$ and $\alpha\geq- 1/2$, that
\[
\mathbb{P}\bigl(\cosh(R_t) \leq z\bigr) = \mathbb{P} \biggl(
X_1 \leq\frac
{z-1}{2A_{t/4}^{(2\alpha+1)}} \biggr).
\]
Hence and from the form of the density of a squared Bessel
process (see \cite{RY}, Chapter XI, Cor.~1.4), we have for $x=0$ and
$\alpha\geq- 1$
\[
\mathbb{P}\bigl(\cosh(R_t) \in \mathrm{d}z\bigr) = \biggl(\frac{1}{4}
\biggr)^{\alpha+1} \frac{(z-1)^{\alpha}}{\Gamma(\alpha+1)}\mathbb{E} \biggl[\mathrm{e}^{-(z-1)/(4A_{t/4}^{(2\alpha+1)})}
\frac{1}{
(A_{t/4}^{(2\alpha+1)} )^{\alpha+1}} \biggr]\,\mathrm{d}z ,
\]
that is, \eqref{gestosc-cosh}.

(b) From the properties of a hyperbolic Bessel process for $\alpha\geq
-1/2$ and $x\geq0$ it follows that $R_t \geq0$. Thus, by \eqref{eqpostacR},
\[
R_t \stackrel{\mathrm{(law)}} = \arcosh \bigl(1 + 2 A_{t/4}^{(2\alpha
+1)}X_1
\bigr),
\]
so
\[
\mathbb{P}(R_t \leq z) = \mathbb{P} \biggl( X_1 \leq
\frac{ \cosh(z)-1}{2
A_{t/4}^{(2\alpha+1)} } \biggr).
\]
Using this equality and again the form of the density of a squared Bessel
process, we have, for $x=0$ and $\alpha\geq- 1/2$,
\begin{eqnarray*}
&&\mathbb{P}(R_t \in \mathrm{d}z)
\\
&&\quad= \biggl(\frac{1}{4} \biggr)^{\alpha+1} \frac{(\cosh(z)-1)^{\alpha}\sinh(z)}{\Gamma(\alpha+1)}\mathbb{E}
\biggl[\mathrm{e}^{-(\cosh(z)-1)/A_{t/4}^{(2\alpha+1)}}\frac{1}{
(A_{t/4}^{(2\alpha+1)} )^{\alpha+1}} \biggr]\,\mathrm{d}z,
\end{eqnarray*}
and the proof is complete.
\end{pf}
%
\begin{remark}
Using the explicit form of the density of $A_t^{(2\alpha+1)}$ and
Proposition \ref{zero}, we can obtain the integral form of the density
of $R_t$ for $\alpha\geq- 1/2$ and $R_0=0$. This formula is new
and differs significantly from the density obtained by Borodin (see
formula (5.4) in \cite{B}).
\end{remark}
The next two facts follow immediately from Theorem \ref{DHB} and in
slightly different forms can be found in \cite{MatII}.

\begin{corollary} For any $z\geq1$, we have
\[
\mathbb{E} \biggl[\frac{1}{\sqrt{A_t}}\exp \biggl({-\frac{(\cosh
(z)-1)}{4A_t}} \biggr)
\biggr] = \sqrt{\frac{2}{ t}}\frac{\sqrt{\cosh(z)-1}}{\sinh(z)} \exp \biggl(-
\frac
{z^2}{8t} \biggr).
\]
\end{corollary}
\begin{pf} This follows from Proposition \ref{zero} for $x=0$ and
$\alpha= -\frac{1}{2}$ and the form of the density of $\cosh(B_t)$.
\end{pf}
%
\begin{proposition} For every $\lambda> 0$, we have
\[
\mathbb{E}A_t^{\lambda} = \frac{\mathbb{E}(\cosh(B_{4t})-1)^{\lambda
}}{2^{\lambda} \mathbb{E}(B_1)^{2\lambda}},
\]
where $B$ is a standard Brownian motion.
\end{proposition}
\begin{pf} This follows from Theorem \ref{DHB} for $x=0$ and $\alpha
= - 1/2$.
\end{pf}
In the next proposition, we deduce from Theorem \ref{DHB} an
alternative purely
probabilistic proof of Bougerol's identity (see
\cite{BB}, \cite{AMY}, Proposition 6.1,
or \cite{MatII}, Section 3). For completeness, we also present the
generalized version of Bougerol's identity which is exactly Proposition
4 in \cite{AG}; in our proof we use similar ideas as in \cite{Comoyo}.
An interesting discussion of Bougerol's identity and its consequences
can also be found in \cite{ADY}.
%
\begin{proposition} \label{tw-genB}
Let $B$ and $W$ be two independent standard Brownian motions.

\emph{(a) (Bougerol's identity)} For $t>0$,
%
\begin{equation}
\label{eqBouger} \sinh(B_t) \stackrel{\mathrm{(law)}} = W_{A_t}.
\end{equation}

\emph{(b)} (\cite{AG}) If $x\in\mathbb{R}$, then
%
\begin{equation}
\sinh(x+ B_t) \stackrel{\mathrm{(law)}} = \sinh(x)\mathrm{e}^{B_t} +
W_{A_t}.
\end{equation}
\end{proposition}
\begin{pf}
(a)
\[
(W_{A_t})^2 \stackrel{\mathrm{(law)}} = A_t(W_1)^2
\stackrel{\mathrm{(law)}} = \frac{\cosh(B_{4t})-1}{2} = \frac{\mathrm{e}^{2 \cdot(1/2)B_{4t}} + \mathrm{e}^{-2
\cdot(1/2)B_{4t}} - 2 }{4} \stackrel{\mathrm{(law)}}
=\bigl(\sinh(B_t)\bigr)^2,
\]
where in the second equality we use Theorem \ref{DHB} with $x=0$ and
$\alpha= - 1/2$. Hence, using the fact that $W_{A_t}$ and
$\sinh(B_t)$ are symmetric random variables we
obtain \eqref{eqBouger}.

(b) By the same arguments as in (a), we conclude that
\[
\bigl| \sinh(x+ B_t) \bigr| \stackrel{\mathrm{(law)}} = \bigl| \sinh(x)\mathrm{e}^{B_t} +
W_{A_t}\bigr|
\]
%
and we can expect that we can skip the absolute
value signs. Indeed, define
\[
M_t = \mathrm{e}^{B_t} \biggl(\sinh(x) + \int_0^t\mathrm{e}^{-B_u}\,\mathrm{d}W_u
\biggr).
\]
Observe, by Proposition 2.1 in \cite{DGY}, that
\[
 \sinh(x)\mathrm{e}^{B_t} + W_{A_t} \stackrel{\mathrm{(law)}} =
\sinh(x)\mathrm{e}^{B_t} + \mathrm{e}^{B_t}\int_0^t\mathrm{e}^{-B_u}\,\mathrm{d}W_u
= M_t .
\]
It is easy to check that $M$ is a diffusion with the generator
%
\begin{equation}
\mathcal{A}_M = \frac12 \bigl(x^2 + 1\bigr)
\frac{\mathrm{d}^2}{\mathrm{d}x^2} + \frac12 x\frac{\mathrm{d}}{\mathrm{d}x}.
\end{equation}
It is also evident that the SDE corresponding to the above generator has
a unique strong solution. To finish
the proof it is enough to observe that $\widehat{M}_t :=
\sinh(x+B_t)$ is a diffusion such that the generators of $M$ and
$\widehat{M}$ are equal on $C_c^2$, so $M \stackrel{\mathrm{(law)}} =
\widehat{M}$.
\end{pf}

A simple but interesting consequence of Proposition \ref{tw-genB} is
the following
proposition.
%
\begin{proposition}\label{expsqra}
Let $B$ be a standard Brownian motion.
For any $\lambda\geq0$ and $t\geq0$ we have
%
\begin{equation}
\mathbb{E}(1+2\lambda A_t)^{-1/2} = \mathbb{E}\mathrm{e}^{-\lambda\sinh^2(B_t)}.
\end{equation}
\end{proposition}
\begin{pf} Let $B$ and $W$ be independent standard Brownian motions.
Using \eqref{eqBESQL},
we have
\[
\mathbb{E}\mathrm{e}^{-\lambda W^2_{A_t}} = \mathbb{E}\mathbb{E} \bigl(\mathrm{e}^{-\lambda W^2_{A_t}}
\vert
\sigma(B_u, u\leq t) \bigr) = \mathbb {E}(1+2\lambda
A_t)^{-1/2}.
\]
The assertion follows from Bougerol's identity.
\end{pf}

In the next proposition, we deduce from Theorem \ref{tw-genB} a
simple form of the characteristic function of the vector $(\mathrm{e}^{B_t},
W_{A_t})$.
%
\begin{proposition} Let $B$ and $W$ be two independent standard
Brownian motions. Then, for $u \in\mathbb{R}$ and $v \neq0$,
%
\begin{equation}
\label{f-char} \mathbb{E}\mathrm{e}^{\mathrm{i}u\mathrm{e}^{B_t} + \mathrm{i}v W_{A_t}} = \mathbb{E}\mathrm{e}^{\mathrm{i}v\sinh(\arsinh(u/v) +B_t)}.
\end{equation}
\end{proposition}
\begin{pf} From Theorem \ref{tw-genB}, we infer that
\[
\frac{u}{v}\mathrm{e}^{B_t} + W_{A_t} \stackrel{\mathrm{(law)}} =
\sinh\bigl(\arsinh(u/v) +B_t\bigr),
\]
which implies \eqref{f-char}.
\end{pf}

Now we establish a representation
of the law of the process $\sinh(R)$, where $R$ is a hyperbolic
Bessel process of index $\alpha= 0$, in terms of functionals of independent
Brownian motions.
%
\begin{theorem} \label{tw-23}
Let $x \geq0$ and $R$ be a hyperbolic Bessel process of index $\alpha
=0$ starting
from $x$. Then
%
\begin{eqnarray}
&&\bigl(\sinh(R_t),t\geq0\bigr)\nonumber\\[-8pt]\\[-8pt]
&&\quad\stackrel{\mathrm{(law)}}=
\biggl(\mathrm{e}^{-B_t + t/2}\biggl( \biggl(\sinh(x) + \int_0^t\mathrm{e}^{B_u -u/2}\,\mathrm{d}V_u
\biggr)^2+ \biggl(\int_0^t\mathrm{e}^{B_u -u/2}\,\mathrm{d}Z_u
\biggr)^2 \biggr)^{1/2}, t\geq0\biggr),\qquad
\nonumber
\end{eqnarray}
where $B,V,Z$ are three independent standard Brownian motions.
\end{theorem}
\begin{pf} Define $\xi_t= \sinh^2(R_t)$. Then the diffusion $\xi$
satisfies the SDE
\[
\mathrm{d}\xi_t = 2 \sqrt{\xi_t^2 +
\xi_t}\, \mathrm{d}B_t + (2+3\xi_t) \,\mathrm{d}t ,
\]
and the
generator of $\xi$ is
%
\begin{equation}
\label{xigen} \mathcal{A}_{\xi} = 2\bigl(x^2 + x\bigr)
\frac{\mathrm{d}^2}{\mathrm{d}x^2} + (2 + 3x)\frac{\mathrm{d}}{\mathrm{d}x}.
\end{equation}
Denote the diffusion and drift coefficients of $\xi$ by
$\sigma(x) = 2\sqrt{x^2 +x}$ and $\mu(x) = 2 + 3x$. Observe
that for any $x\geq0$ and $y\in(x-1,x+1)$,
\begin{eqnarray*}
\bigl(\sigma(x) -\sigma(y) \bigr)^2 & 
\leq&4
\bigl(\bigl|x^2-y^2\bigr| + |x-y| \bigr) = 4|x-y|\bigl(|x+y| +1\bigr)
\leq4|x-y|\bigl(2|x| +2\bigr)
\\
&\leq&4|x-y|\bigl(\sigma(x) +2\bigr) \leq4|x-y|\bigl(\sigma^2(x)/2
+5/2\bigr).
\end{eqnarray*}
The uniqueness of solution for the corresponding SDE follows now from
\cite{RY}, Chapter IX, Ex.~3.14. Thus there is a unique solution of
the martingale problem induced by $\mathcal{A}_{\xi}$.

For a standard Brownian motion $B$ we define, as previously, $Y_t =
\mathrm{e}^{B_t-t/2}$.
We now consider a SDE of the form
%
\begin{equation}
\label{svp} \mathrm{d}X_t = 2\sqrt{X_t}Y_t\,\mathrm{d}W_t
+ 2Y_t^2\,\mathrm{d}t,
\end{equation}
where $X_0 = \sinh^2(x)$ and $W$ is a standard Brownian motion
independent of $B$. Observe, using the It\^o lemma, that the
process $\psi= \frac{X}{Y^2}$ is a diffusion with the generator
$\mathcal{A}_{\psi}$ having the same form as $\mathcal{A}_{\xi}$ on
$C_c^2$. Hence, by uniqueness of solution of the martingale
problem induced by $\mathcal{A}_{\xi}$, the
processes $|\sinh(R)|$ and $|\sqrt{X}/Y|$ have the same law. To skip
the absolute value signs we note that both processes are nonnegative.
Moreover, the SDE (\ref{svp}) has a unique weak solution. Indeed,
using the change of time $\tau_t = \inf\{s\geq0\dvt  \int_0^sY_u^2\,\mathrm{d}u
\geq t\}$ we find that the process $X_{\tau_{\cdot}}$ is the square
of a 2-dimensional Bessel process, and it is not
difficult to see that the process $X$ inherits ``good'' properties of
the square of a 2-dimensional Bessel process. To finish the proof, we
observe that the unique weak solution of (\ref{svp}) can be written
as
%
\begin{equation}
\label{solx} X_t = \biggl(\sinh(x) + \int_0^tY_u\,\mathrm{d}V_u
\biggr)^2+ \biggl(\int_0^tY_u\,\mathrm{d}Z_u
\biggr)^2.
\end{equation}
To verify that (\ref{solx}) satisfies (\ref{svp}), it is enough to
observe that $W$
defined as
%
\begin{equation}
W_t := \int_0^t
1_{\{\Lambda_s \not= 0 \}} \frac{(\sinh(x) + \int_0^sY_u\,\mathrm{d}V_u)\,\mathrm{d}V_s +
(\int_0^sY_u\,\mathrm{d}Z_u)\,\mathrm{d}Z_s}{\sqrt{(\sinh(x) +
\int_0^sY_u\,\mathrm{d}V_u)^2+(\int_0^sY_u\,\mathrm{d}Z_u)^2}},
\end{equation}
where $\Lambda_s=(\sinh(x) +
\int_0^sY_u\,\mathrm{d}V_u)^2+(\int_0^sY_u\,\mathrm{d}Z_u)^2 $, is a standard Brownian
motion. A simple use of the It\^o lemma finishes the proof.
\end{pf}
It turns out that the methods of the
proof of Theorem \ref{tw-23} also allow us to find a representation
of the law of the hyperbolic sine of a hyperbolic Bessel process.
%
\begin{theorem} \label{thmsinhR}
Let $x \geq0$ and $R$ be a hyperbolic Bessel process of index $\alpha
\geq-1/2$ starting
from~$x$. Then
%
\begin{equation}
\label{dsinh} \bigl(\sinh(R_t), t\geq0 \bigr) \stackrel{\mathrm{(law)}}= (
\sqrt{X_t}/Y_t, t\geq0 ),
\end{equation}
where $Y_t = \mathrm{e}^{B_t -(\alpha+1/2)t}$, $X_t$ satisfies the SDE
%
\begin{equation}
\label{dex-1} \mathrm{d}X_t = 2\sqrt{X_t}Y_t\,\mathrm{d}W_t
+ 2(1+\alpha)Y_t^2\,\mathrm{d}t,
\end{equation}
$X_0 = \sinh^2(x)$ and $B, W$ are two independent standard Brownian motions.
\end{theorem}
\begin{pf} The proof follows the lines of the previous one.
\end{pf}

\begin{theorem} \label{thmroz-sinhR}
Let $Y$ be given by $Y_t = \mathrm{e}^{B_t -(\alpha+1/2)t}$, where $B$ is a
standard Brownian motion. For any $w\geq0, t\geq0$,
%
\begin{equation}
\bigl(\sinh(R_t), t\geq0 \bigr) \stackrel{\mathrm{(law)}}=
\bigl((Y_t)^{-1}S_{\int_0^tY_u^2\,\mathrm{d}u}, t\geq0 \bigr),
\end{equation}
where $S$ is a Bessel process of dimension $2(1+\alpha)$ independent of
$B$, and $S_0 = \sinh(x)$.
\end{theorem}
\begin{pf} Consider the process $X$ having dynamics given by \eqref
{dex-1} with $B$ and $ W$ independent standard Brownian motions. Let
$\tau_t = \inf\{s\geq0\dvt  \int_0^sY_u^2\,\mathrm{d}u \geq t\}$.
Observe that the process $X_{\tau_t}$ is
the square of a $2(1+\alpha)$-dimensional Bessel process.
Observe also that conditionally, under the knowledge of the
trajectory of $(Y_s, s\leq t)$, the process $X_{\tau_t}$ is still the
square of a $2(1+\alpha)$-dimensional Bessel process.
Hence, the assertion follows from Theorem \ref{thmsinhR}.
\end{pf}

\begin{proposition}Let $k\in\mathbb{N}$, $x\geq0$, and let $Y$ be a
given continuous process such that
$\mathbb{E}\int_0^tY_u^2\,\mathrm{d}u < \infty$ for any $t\geq0$. Then the unique
solution of the SDE
%
\begin{equation}
\label{gsde} \mathrm{d}X_t = 2\sqrt{X_t}Y_t\,\mathrm{d}W_t
+ kY_t^2\,\mathrm{d}t,
\end{equation}
where $W$ is a standard Brownian motion independent of $Y$ and $X_0
=x$, is
%
\begin{equation}
\label{gsol} X_t = \biggl(\sqrt{x} +\int_0^tY_u\,\mathrm{d}W^1_u
\biggr)^2 + \sum_{i=2}^k
\biggl(\int_0^tY_u\,\mathrm{d}W^i_u
\biggr)^2,
\end{equation}
where $W^i$, $i=1, \ldots,k$, are independent standard Brownian
motions.
\end{proposition}
\begin{pf} As previously, the uniqueness of solution of (\ref{gsde})
follows from the fact that the standard change of time ($\tau_t=
\inf\{u\dvt \int_0^uY_s^2\,\mathrm{d}s > t\}$) results in $X$ becoming the square of a
$k$-dimensional Bessel process. We complete the proof by the direct
checking that $X$ defined by (\ref{gsol}) satisfies (\ref{gsde}),
using analogous arguments to those in the proof of Theorem
\ref{thmroz-sinhR}.
\end{pf}

\subsection{Hyperbolic Bessel processes with stochastic time}
%
\begin{proposition} \label{ConP}
Let $T_{\delta}$ be a random variable with exponential
distribution with parameter $\delta>0$. Assume that
$T_{\delta}$ is independent of a standard Brownian motion $B$.
Then for a hyperbolic Bessel process of the form \eqref{eqhbp1}
with $\alpha\geq-1/2$,
we have
%
\begin{equation}
\mathbb{E}\exp (-\lambda\cosh R_{T_{\delta}} ) =\int_0^{\infty}
\int_0^{\infty}\mathrm{e}^{- \lambda u \cosh(x)-({1}/{2})v}p^{\gamma}(u,1,y)\,\mathrm{d}y\,\mathrm{d}u,
\end{equation}
where for $\gamma= \sqrt{2\delta+ (\alpha+\frac{1}{2})^2}$,
%
\begin{equation}
p^{\gamma}(u,1,y) = \frac{\delta}{y^{1/2 - \alpha
}u}\mathrm{e}^{-(y^2+1)/(2u)}I_{\gamma}
\biggl(\frac y u \biggr)
\end{equation}
and $I_{\gamma}$ is a modified Bessel function.
\end{proposition}
\begin{pf} We use Theorem \ref{twHB} and the result of
Matsumoto-Yor \cite{MatII}, Theorem 4.11.
\end{pf}

\section*{Acknowledgements}
Research supported in part by Polish MNiSW Grant N N201 547838. We would like to
thank the anonymous referee for valuable remarks and suggestions.



\printhistory

\end{document}